\documentclass[11pt]
{amsart}
\usepackage{amssymb,amsmath,amsthm,amsfonts,amsopn,url,color,hyperref,
mathtools,microtype,MnSymbol,dutchcal,
mathrsfs}
\usepackage{scalerel}
\usepackage{stackengine,wasysym}
\usepackage[shortlabels]{enumitem}

\usepackage[normalem]{ulem}
\usepackage[all]{xy}
\input xy
\xyoption{all}
\usepackage{amscd}
\usepackage{soul}

\theoremstyle{plain}
\newtheorem{thm}{Theorem}[section]

\newtheorem*{mainthm}{Main Theorem}
\newtheorem{prop}[thm]{Proposition}

\newtheorem{cor}[thm]{Corollary}

\theoremstyle{definition}
\newtheorem{defn}[thm]{Definition}
\newtheorem*{defn*}{Definition}
\newtheorem*{question*}{Question}
\newtheorem{question}{Question}
\newtheorem{example}[thm]{Example}
\newtheorem*{example*}{Example}
\newtheorem{rem}[thm]{Remark}
\newtheorem*{rem*}{Remark}

\newcommand{\field}[1]{\mathbb{#1}}
\newcommand{\N}{\field{N}}
\newcommand{\Z}{\field{Z}}

\newcommand{\ra}{\rightarrow}

\newcommand{\be}{\begin{enumerate}}
\newcommand{\ee}{\end{enumerate}}

\newcommand{\li}
 {\leftfootline}

\renewcommand{\phi}{\varphi}
\DeclareMathOperator{\Frac}{Frac}

\let\int\relax
\DeclareMathOperator{\int}{i}

\author{Neil Epstein}
\address{Department of Mathematical Sciences \\ George Mason University \\ Fairfax, VA  22030}
\email{nepstei2@gmu.edu}

\title{The unit fractions from a Euclidean domain generate a DVR}

\date{May 25, 2023}
\begin{document}
\maketitle
\begin{abstract}
Let $D$ be a Euclidean domain, with fraction field $K$.  Let $R(D)$ be the subring of $K$ generated by the reciprocals of the nonzero elements of $D$.  The main theorem states that if $R(D) \neq K$, then $R(D)$ is a rank 1 discrete valuation ring that contains a field consisting of the units of $D$ along with $0$.  Connections are made to ideas from medieval Italian mathematics.
\end{abstract}


\section{Introduction}

Let $D$ be an integral domain, wth fraction field $K$.  Let $R(D)$ be the subring of $K$ generated by all the \emph{reciprocals} of nonzero elements of $D$.  What can we say about $R(D)$?

The above question is raised implicitly in \cite{GLO-Egypt}, where an integral domain $D$ is called \emph{Egyptian} if any element of $d$ (equiv. any element of $K$) can be written as a sum of reciprocals of distinct elements of $D$.  Since it is also shown in \cite[Theorem 2]{GLO-Egypt} that any element that can be written as a sum of reciprocals from $D$ can also be written as a sum of distinct reciprocals from $D$, we have that $D$ is Egyptian if and only if $R(D)=K$.

In \cite{GLO-Egypt}, it is shown that $\Z$ is Egyptian, but $k[x]$ is not.  This is at first surprising, as both are Euclidean domains.  Thus, the current article seeks to explain the distinction between the two cases.  I show here that for any Euclidean domain $D$, $R(D)$ is either a field or a DVR, and unless the units of $D$ along with $0$ form a field (as is the case with $k[x]$, but not with $\Z$), $D$ is Egyptian (i.e. $R(D)$ is a field).  That is,
\begin{mainthm}[See Theorem~\ref{thm:main}]
    Let $D$ be a Euclidean domain, with fraction field $K$, and let $R$ be the subring of $K$ generated by the set $\{1/d\mid 0\neq d \in D\}$.  Then exactly one of the following is true: \begin{enumerate}
        \item $D$ is Egyptian -- i.e. $R = K$.
        \item $R$ is a (nontrivial) DVR, and the units of $D$ along with $0$ form a field.
    \end{enumerate}
\end{mainthm}

Two corollaries are as follows: \begin{itemize}
    \item\ (Corollary~\ref{cor:EuEg}): Any Euclidean domain that does not contain a field is Egyptian.
    \item\ (Proposition~\ref{pr:EgBo}): Let $D$ be an Egyptian domain, and $x$ an indeterminate.  Then $R(D[x])$ is a DVR.
\end{itemize}

\section{Bonaccian domains, reciprocal complements, and DVRs}

\begin{defn}
Let $D$ be an integral domain, and $K$ its fraction field.  A \emph{unit fraction (from $D$)} is an element of the form $1/d$, $d \in D \setminus \{0\}$.  An element $\alpha \in K$ is \emph{$D$-Egyptian}, or just \emph{Egyptian} if the context is clear (as in \cite{GLO-Egypt}) if it is a sum of reciprocals of distinct elements of $D$ -- i.e. a sum of distinct unit fractions from $D$.  One says that $D$ is \emph{Bonaccian} if for all nonzero $\alpha \in K$, either $\alpha$ or $\alpha^{-1}$ is $D$-Egyptian.  If \emph{every} nonzero element of $K$ (equivalently of $D$) is $D$-Egyptian, we say $D$ is \emph{Egyptian} (as in \cite{GLO-Egypt}).

Set $R(D) := $ the subring of $K$ generated by the unit fractions with denominators in $D$.  We call $R(D)$ the \emph{reciprocal complement} of $D$.  
\end{defn}

\begin{rem}\label{rem:RD}
Note that every element of $R(D)$ is a finite sum of unit fractions from $D$, and that the fraction field of $R(D)$ is $K$.
%
\end{rem}

\begin{rem}
The term ``Bonaccian" is in honor of Fibonacci, also known as Leonardo of Pisa, who proved that the integers are Bonaccian.  In particular, he showed that any positive \emph{proper} fraction may be written as a sum of \emph{distinct} unit fractions.  The same result then trivially follows for negative proper fractions (by negating all the denominators).  The proof of Theorem~\ref{thm:lowerdegree} is based on Fibonacci's algorithm.
\end{rem}


\begin{prop}\label{pr:Bonval}
$D$ is Bonaccian if and only if $R(D)$ is a valuation ring, whereas $D$ is Egyptian if and only if $R(D) = \Frac(D)$ if and only if $D \subseteq R(D)$.
\end{prop}

\begin{proof}
Recall \cite[Theorem 2]{GLO-Egypt} that any finite sum of unit fractions from $D$, even if not all the reciprocals are distinct, is Egyptian.  Hence by Remark~\ref{rem:RD}, the nonzero elements of $R(D)$ are  exactly the $D$-Egyptian elements of $K$.  Then all statements of the proposition follow from the definitions.
%
\end{proof}


As there are varying definitions of Euclidean domains and Euclidean functions in the literature, we fix the following definitions below.

\begin{defn}
    Let $D$ be an integral domain.  A \emph{Euclidean function} on $D$ is a function $f: D\setminus \{0\} \ra \Z$ such that for any nonzero $a,b\in D$, \begin{enumerate}
        \item $f(ab) \geq f(a)$, and
        \item There exist $q,r \in D$ such that $b=aq+r$ and either $r=0$ or $f(r) < f(a)$.
    \end{enumerate}
    A \emph{Euclidean domain} is an integral domain that admits a Euclidean function.
\end{defn}

Here are some well-known facts about Euclidean domains and functions.

\begin{prop}\label{pr:Euclid}
Let $D$ be a Euclidean domain, with Euclidean function $f$.  Then: \begin{enumerate}
    \item For any unit $u$ of $D$, $f(u)=f(1)$.
    \item For any nonzero $a,b\in D$ such that $a$ is not a unit, we have $f(ab)>f(b)$.  In particular, $f(a)>f(1)$.
    \item For any nonzero $a,b,c\in D$ such that $a$ is not a unit, there is some positive integer $n$ such that $f(a^nb) > f(c)$.
 \end{enumerate}
\end{prop}

\begin{proof}
(1)    Let $u,v\in D$ with $uv=1$.  Then $f(1) =f(uv) \geq f(u)$, but also $f(u) = f(1\cdot u) \geq f(1)$.

(2)    Now let $a,b \in D\setminus \{0\}$ such that $a$ is not a unit. Then $f(ab) \geq f(b)$.  If equality holds, there exist $q,r \in D$ with $b=abq + r$ and either $r=0$ or $f(r) < f(ab)=f(b)$. In the first case, we have $b=abq$, so that $b(1-aq) = 0$, and since $b\neq 0$ and $D$ is a domain, we have $aq=1$, contradicting the assumption that $a$ is a nonunit.  In the second case, we have $r=b(1-aq)$, so that $f(b) > f(r) = f(b(1-aq)) \geq f(b)$, also a contradiction. Thus, $f(ab) > f(b)$.  The last statement holds by specializing to the case $b=1$.

(3)    Finally, let $a,b,c \in D \setminus \{0\}$.   For any nonzero $x\in D$, from (2) we have $f(ax) > f(x)$, so that since $f$ takes values in $\Z$, $f(ax) \geq f(x) + 1$. It follows by induction that for any positive integer $n$, $f(a^n b) \geq f(b)+n$.  Thus, if $n$ is chosen so that $n > f(c) - f(b)$, we have $f(a^nb) \geq f(b) + n > f(b) + (f(c) - f(b)) = f(c).$
\end{proof}

\begin{thm}\label{thm:lowerdegree}
Any Euclidean domain $D$ is Bonaccian.  In particular, if $f$ is a Euclidean function on $D$, then for nonzero $a,b \in D$, $a/b$ is Egyptian provided that $f(a) \leq f(b)$.
\end{thm}

\begin{proof}
Assume $a,b$ are nonzero elements of $D$, with $f(a) \leq f(b)$.  By \cite[Theorem 2]{GLO-Egypt}, it suffices to show that $a/b \in R(D)$.

The proof is by induction on the number $f(a)$.  By Proposition~\ref{pr:Euclid}, the smallest possible value occurs precisely when $a$ is a unit of $D$, in which case $f(a)=f(1)$.  But then $a^{-1} \in D$, so that $\frac ab = \frac 1 {a^{-1}b}\in R(D)$.

For the inductive step, assume $f(a)>f(1)$.  There exist $q,r \in D$ such that $b=qa+r$ and either $r=0$ or $f(r)<f(a)$.  If $r=0$ then $aq=b\neq 0$, so $q \neq 0$.  Thus, we can divide by $bq$ and see that $\frac ab=\frac 1q \in R(D)$.  If $r\neq 0$, so that $f(r)<f(a)$, then since $f(a) \leq f(b)$, it follows that $f(r)<f(b)$, so that again we have $q \neq 0$ (as otherwise $b=r$).  Now divide through by $bq$ to get $\frac 1q = \frac ab + \frac r{bq}$.
That is, $\frac ab = \frac 1q + \frac r {-bq}$.  But since $f(r)<f(a)$, $f(a) \leq f(b)$, and $f(b) \leq f(b(-q)) = f(-bq)$ by Proposition~\ref{pr:Euclid}, we have $f(r) < f(-bq)$, so by induction we have that $\frac r{-bq} \in R(D)$.  Thus $\frac ab$, being a sum of two elements of $R(D)$, is itself in $R(D)$.
%
\end{proof}

\begin{rem}
The above proof is based on Fibonacci's classic explanation in the book \emph{Liber Abacci} (1202) of the fact that any positive rational number can be written as a sum of reciprocals of distinct positive integers.  For a translation of the relevant portion of his book, see \cite{DuGr-FibE}.
\end{rem}

\begin{example}\label{ex:notBon}
Let $D = k[x,y]$, where $k$ is a field.  Then I claim neither $x/y$ nor $y/x$ is $D$-Egyptian.  By symmetry, I need only show that $x/y$ is not $D$-Egyptian.

To see this, suppose we have \[
\frac xy = \frac 1 {f_1} + \cdots + \frac 1 {f_n}
\]
with $f_1, \ldots, f_n \in k[x,y]$ nonzero polynomials.  Then each $f_i/y$ is a nonzero element of the ring $k(y)[x]$, where $k(y)$ is the fraction field of $k[y]$.  Thus, \[
x=\frac 1 {f_1/y} + \cdots + \frac 1 {f_n/y},
\]
so that $x$ is an Egyptian element of $k(y)[x]$, contradicting \cite[Proposition 1 and its proof]{GLO-Egypt}.

Recall \cite[Definition 3.1]{nme-Edom} that an integral domain $D$ is \emph{locally Egyptian} if there exist nonzero $c_1, \ldots, c_n \in D$ that generate the unit ideal and such that $D[1/c_i]$ is Egyptian for each $1\leq i \leq n$.  I showed \cite[Corollary 3.11]{nme-Edom} that any finitely generated $k$-algebra domain is locally Egyptian.
Thus, the above calculation shows that not all locally Egyptian domains are Bonaccian.
\end{example}

We are ready to state the main theorem:
\begin{thm}[Main Theorem]\label{thm:main}
Let $D$ be a Euclidean domain that is not Egyptian.  Then $R(D)$ is a DVR, and the units of $D$ together with $0$ form a field.
\end{thm}

First, we will prove the part about the units comprising the nonzero elements of a field.
\begin{prop}\label{pr:units}
Let $(D,f)$ be a Euclidean domain, $K = \Frac D$, and $R=R(D)$.  If $D \nsubseteq R$ (i.e. $R \neq K$) then $R \cap D = \{$the units of $D\} \cup \{0\}$.  In particular, the units of $D$ along with zero form a \emph{field} contained in $D$.
\end{prop}

\begin{rem}\label{rem:ua}
    It turns out that the class of rings $R$ such that the units and nilpotent elements comprise a subring of $R$ is interesting.  See \cite{nmeSh-unitadd} for a study of such rings, which are called \emph{unit-additive}.  In those terms, Proposition~\ref{pr:units} implies that any Euclidean domain must be either Egyptian or unit-additive.
\end{rem}
\begin{proof}[Proof of Proposition~\ref{pr:units}]
Let $X := \{f(d) \mid d \in D \setminus R\}$.  Then $X$ is a nonempty set of integers, bounded below by the integer $f(1)$ due to Proposition~\ref{pr:Euclid}, hence $X$ admits a minimal member $m$.  Choose $y\in D \setminus R$ such that $f(y)=m$.

Now let $x$ be any nonzero nonunit of $D$.  Suppose $x\in R$.  Then by Proposition~\ref{pr:Euclid}, there is some $n\in \N$ such that $f(x^n) =f(x^n \cdot 1)\geq f(y)$.  Then there exist $q,r \in D$ such that $x^n = qy + r$ and either $r=0$ or $f(r)<f(y)$.  In either case, it follows that $r\in R$ and $q\neq 0$.  Since $x^n, r, 1/q \in R$, it follows that $y = \frac 1q (x^n-r) \in R$, which contradicts our choice of $y$.

Thus, no nonzero nonunits of $D$ are in $R$.  But clearly $0\in R$ and all units of $D$ are in $R$.

The final statement then holds because $D \cap R$ must be a subring of $K$, and for any $0\neq d \in D \cap R$, we have that $d$ is a unit of $D$, whence $d^{-1} \in D$, so that since $d^{-1} \in R$ (by definition of $R$), we have $d^{-1} \in D \cap R$.
\end{proof}

\begin{cor}\label{cor:EuEg}
Any Euclidean domain that doesn't contain a field is Egyptian.
\end{cor}

\begin{proof}
    Combine Propositions~\ref{pr:Bonval} and \ref{pr:units}.
\end{proof}

\begin{proof}[Proof of Theorem~\ref{thm:main}]
Let $y\in D$ be a nonzero nonunit such that $f(y) = \min \{f(x) \mid x \text{ is a nonzero nonunit of } D\}$. Set $\pi := 1/y$ and $R := R(D)$.  Note that $\pi \in R$ (as it is the reciprocal of an element of $D$), and since $y=\pi^{-1} \notin R$ by Proposition~\ref{pr:units}, $\pi$ is a nonunit of $R$.

Let $a,b \in D \setminus \{0\}$ such that $\alpha := a/b$ is a nonunit of $R$.  Then by Theorem~\ref{thm:lowerdegree}, $f(a) < f(b)$ (since if $f(b)\leq f(a)$ then $b/a = \alpha^{-1} \in R$).  Then there exist $q,r \in D$ with $b=aq+r$ and either $f(r)<f(a)$ or $r=0$. In either case it follows that $b\neq r$, so that $q\neq 0$.

If $q\in R$, then since $\frac ra \in R$ by Theorem~\ref{thm:lowerdegree}, we have $\alpha^{-1} =\frac ba = q + \frac ra \in R$.  But this contradicts the assumption that $\alpha$ is not a unit of $R$.  Thus, $q\notin R$.  In particular, $q$ must be a nonunit of $D$, so by choice of $y$ we have $f(q) \geq f(y)$. Therefore, there exist $c,d \in D$ with $c\neq 0$ and either $d=0$ or $f(d)<f(y)$ (so that in either case, $d\in R$) such that $q=yc+d$.  Then \[
\alpha = \frac ab = \frac 1q - \frac r{qb} = \frac 1q \left(1-\frac rb\right) \in \frac 1q R,
\]
since either $r=0$ or $f(r) < f(b)$, and \[
\frac 1q = \frac 1y \left( \frac 1c - \frac d {qc}\right) \in \frac 1yR = \pi R.
\]
Thus, $\alpha \in \pi R$, so that $\pi R$ consists of \emph{all nonunits} of $R$.

Since $R$ is a valuation ring (by Proposition~\ref{pr:Bonval} and Theorem~\ref{thm:lowerdegree}) with principal maximal ideal $\pi R$ (by the above), to prove it is a DVR it suffices to show that $\bigcap_{n \in \N} \pi^n R = 0$ \cite[p. 99, exercise 4]{AtMac-ICA}.  For this, let $\alpha = \frac ab$ (with $a,b \in D$) be an arbitrary nonzero element of $R$.  By Proposition~\ref{pr:Euclid}, there is some $n\geq 0$ such that $f(y^na) \geq f(b)$.  
Then by Theorem~\ref{thm:lowerdegree}, $\pi^n \alpha^{-1} = \frac b {y^na} \in R$, whence $\pi^n \in \alpha R$.  Now suppose $\alpha \in \pi^{n+1}R$.  Then $\pi^n \in \pi^{n+1}R$, say $\pi^n = \pi^{n+1}\beta$, $\beta \in R$ so that $\pi^n (1-\pi \beta) = 0$.  Since $\pi^n \neq 0$ and $R$ is a domain, $\pi\beta = 1$.  But then $\pi$ is a unit of $R$, a contradiction. Thus, $\alpha \notin \pi^{n+1}R$, so that $\alpha \notin \bigcap_k \pi^k R$.
\end{proof}

\begin{example}\label{ex:kxBon}
Recall from \cite[Proposition 1 and its proof]{GLO-Egypt} that if $k$ is a field and $x$ an indeterminate, then $x$ is \emph{not} a sum of unit fractions from $D := k[x]$.  Thus by the main theorem above, $R:= R(D)$ is a DVR.  In fact, we can determine its identity precisely.  
Namely, we clearly have $k[x^{-1}] \subseteq R$, and for any $f(1/x) \in k[x^{-1}] \setminus x^{-1} k[x^{-1}]$, let $d=\deg f$.  Then as $f$ has nonzero constant term, we have $g=x^d f(1/x)$ is a nonzero polynomial in $k[x]$ of degree exactly $d$.  Thus, $\frac 1 {f(1/x)} = \frac {x^d} {g(x)}$ is a fraction where the numerator and denominator both have degree $d$.  But since degree is a Euclidean function on $D$, it follows from Theorem~\ref{thm:lowerdegree} that $\frac 1 {f(1/x)} = \frac {x^d}{g(x)}$ is $D$-Egyptian. Hence $k[x^{-1}]_{(x^{-1})} \subseteq R$.  But there are no rings properly between $k[x^{-1}]_{(x^{-1})}$ and $k(x)$, so that since $R \neq k(x)$, we have $R=k[x^{-1}]_{(x^{-1})}$.
\end{example}


\begin{prop}\label{pr:EgBo}
Let $D$ be an Egyptian domain, let $K$ be its fraction field, and let $x$ be an indeterminate.  Then $R(D[x]) = R(K[x])$.  Hence $D[x]$ is Bonaccian.
\end{prop}

\begin{proof}
Clearly the function $R(-)$ is order-preserving on subrings of $K(x)$.  Thus, for the first statement it is enough to show that $R(K[x]) \subseteq R(D[x])$.

Let $0\neq g \in K[x]$.  Then there is some $d \in D \setminus \{0\}$ such that $dg \in D[x]$.  But since $D$ is Egyptian, there exist nonzero $c_1, \ldots, c_s \in D$ with $d = \sum_{i=1}^s \frac 1{c_i}$.  Thus, \[
\frac 1g = \frac d{dg} = \frac 1{dg} \left(\frac 1{c_1} + \cdots + \frac 1{c_s}\right) \in R(D[x]).
\]
Since $R(K[x])$ is generated by such terms $\frac 1g$, it follows that $R(K[x]) \subseteq R(D[x])$.

For the final statement, Example~\ref{ex:kxBon} shows that $R(K[x]) \cong K[x^{-1}]_{(x^{-1})}$, a DVR.  But then since $R(D[x])= R(K[x])$, $D[x]$ is Bonaccian.
\end{proof}

\section{Some questions}
In the process of this research, some natural questions came up regarding properties of reciprocal complements.  They are enumerated below.

\begin{question}
Is $R(D)$ always local?
\end{question}
The answer to the above is ``yes'' provided $D$ is a Euclidean domain, as then $R(D)$ is either a DVR or a field.

\begin{question}
If $D$ is Noetherian (resp. of finite Krull dimension), is $R(D)$ also?
\end{question}
Again the answer is ``yes'' when $D$ is a Euclidean domain.

\begin{question}
If $D$ is a PID (but not Euclidean), is it Bonaccian?  If so, is $R(D)$ a DVR?
\end{question}
Due to Example~\ref{ex:notBon}, some condition must be imposed.  But perhaps being a PID is sufficient.

\begin{question}
    Suppose $D$ is a Euclidean domain that is unit-additive (see Remark~\ref{rem:ua}).  Can it be Egyptian?
\end{question}
This is a case not covered by the Main Theorem.

\providecommand{\bysame}{\leavevmode\hbox to3em{\hrulefill}\thinspace}
\providecommand{\MR}{\relax\ifhmode\unskip\space\fi MR }
\providecommand{\MRhref}[2]{%
  \href{http://www.ams.org/mathscinet-getitem?mr=#1}{#2}
}
\providecommand{\href}[2]{#2}

\end{document}